\documentclass[11pt, a4paper]{article}

\usepackage{graphicx}
\usepackage{ae}
\usepackage{amsmath}
\usepackage{amssymb}
\usepackage{latexsym}
\usepackage{url}
\usepackage{epsfig}
\usepackage{lineno}
\usepackage{cite}
\usepackage{mathrsfs}
\usepackage{amsfonts}
\usepackage{amsthm}
\usepackage{indentfirst}
\allowdisplaybreaks

\input amssym.def
\input amssym.tex

\def\qed{\hfill$\Box$\vspace{12pt}}

\usepackage{color}

\long\def\delete#1{}

\newcommand{\bmat}[1]{\begin{bmatrix}#1\end{bmatrix}}
\newcommand{\pmat}[1]{\begin{pmatrix}#1\end{pmatrix}}

\newcommand{\be}{\begin{equation}}
\newcommand{\ee}{\end{equation}}
\newcommand{\ben}{\begin{equation*}}
\newcommand{\een}{\end{equation*}}
\newcommand{\bea}{\begin{eqnarray}}
\newcommand{\eea}{\end{eqnarray}}
\newcommand{\bean}{\begin{eqnarray*}}
\newcommand{\eean}{\end{eqnarray*}}

\def\adj{{\rm adj}}

\def\diag{{\rm diag}}

\newtheorem{thm}{Theorem}[section]
\newtheorem{cor}[thm]{Corollary}

\newtheorem{prop}[thm]{Proposition}
\newtheorem{defn}[thm]{Definition}

\numberwithin{equation}{section}

\title{\textbf{Spectra of subdivision-vertex and subdivision-edge joins of graphs}\thanks{This work is supported by the Natural Science Foundation of China (No.11361033).}}

\author{Xiaogang Liu$^{1,2}$\, and\, Zuhe Zhang$^{2,}$\thanks{Corresponding Author. E-mail address: \tt xiaogliu.yzhang@gmail.com (X. Liu), zhang.zuhe@gmail.com (Z. Zhang)}\\
\footnotesize{1. Department of Applied Mathematics, Northwestern Polytechnical University, Xi'an, 710072, Shaanxi, PR China}\\
\footnotesize{2. School of Mathematics and Statistics, The University of Melbourne, Parkville, VIC 3010, Australia}}

\date{}

\begin{document}

\openup 0.5\jot
\maketitle

\begin{abstract}
The subdivision graph $\mathcal{S}(G)$ of a graph $G$ is the graph obtained by inserting a new vertex into every edge of $G$. Let $G_1$ and $G_2$ be two vertex disjoint graphs. The \emph{subdivision-vertex join} of $G_1$ and $G_2$, denoted by $G_1\dot{\vee}G_2$, is the graph obtained from $\mathcal{S}(G_1)$ and $G_2$ by joining every vertex of $V(G_1)$ with every vertex of $V(G_2)$. The \emph{subdivision-edge join} of $G_1$ and $G_2$, denoted by $G_1\underline{\vee}G_2$, is the graph obtained from $\mathcal{S}(G_1)$ and $G_2$ by joining every vertex of $I(G_1)$ with every vertex of $V(G_2)$, where $I(G_1)$ is the set of inserted vertices of $\mathcal{S}(G_1)$. In this paper we determine the adjacency spectra, the Laplacian spectra and the signless Laplacian spectra of $G_1\dot{\vee}G_2$ (respectively, $G_1\underline{\vee}G_2$) for a regular graph $G_1$ and an arbitrary graph $G_2$, in terms of the corresponding spectra of $G_1$ and $G_2$. As applications, these results enable us to construct infinitely many pairs of cospectral graphs. We also give the number of the spanning trees and the Kirchhoff index of $G_1\dot{\vee}G_2$ (respectively, $G_1\underline{\vee}G_2$) for a regular graph $G_1$ and an arbitrary graph $G_2$.

\bigskip


\bigskip

\end{abstract}

\section{Introduction}

All graphs considered in this paper are simple and undirected. Let $G=(V(G),E(G))$ be a graph with vertex set $V(G)=\{v_1,v_2,\ldots,v_n\}$ and edge set $E(G)$. The \emph{adjacency matrix} of $G$, denoted by $A(G)=(a_{ij})_{n\times n}$, is an $n\times n$
symmetric matrix such that $a_{ij}=1$ if vertices $v_i$ and $v_j$
are adjacent and $0$ otherwise. Let $d_{i}=d_G(v_i)$ be the degree of
vertex $v_i$ in $G$ and $D(G)=\diag(d_{1},d_{2},\ldots,d_{n})$ be
the diagonal matrix of vertex degrees. The \emph{Laplacian matrix} and \emph{signless Laplcacian matrix} of $G$ are defined as
$L(G)=D(G)-A(G)$ and $Q(G)=D(G)+A(G)$, respectively. Given an $n \times n$ matrix $M$, denote by
$$\phi(M;x)=\det(xI_n-M),$$ or simply $\phi(M)$, the characteristic polynomial of $M$, where $I_n$ is the identity matrix of size $n$. In particular, for a graph $G$, we called  $\phi(A(G))$ (respectively, $\phi(L(G))$,  $\phi(Q(G))$) the \emph{adjacency} (respectively, \emph{Laplacian, signless Laplacian}) \emph{characteristic polynomial} of $G$ and its roots the \emph{adjacency} (respectively, \emph{Laplacian}, \emph{singless Laplacian}) eigenvalues of $G$.  The adjacency eigenvalues of $G$, denoted by $\lambda_1(G)\geq\lambda_2(G)\geq\cdots\geq\lambda_n(G)$, are called the \emph{$A$-spectrum} of $G$. Similarly, the eigenvalues of $L(G)$ and $Q(G)$, denoted by $0=\mu_1(G)\leq\mu_2(G)\leq\cdots\leq\mu_n(G)$ and  $\nu_1(G)\leq\nu_2(G)\leq\cdots\leq\nu_n(G)$ respectively, are called the \emph{$L$-spectrum} and \emph{$Q$-spectrum} of $G$ accordingly. Two graphs are said to be \emph{$A$-cospectral} (respectively, \emph{$L$-cospectral, $Q$-cospectral}) if they have the same $A$-spectrum (respectively, \emph{$L$-spectrum}, \emph{$Q$-spectrum}). It is well known that graph spectra store a lot of structural information about a graph. See \cite{kn:Cvetkovic95,kn:Cvetkovic10,kn:Brouwer12} and the references therein to know more.

Calculating the spectra of graphs as well as formulating the characteristic polynomials of graphs is
a fundamental and very meaningful work in spectral graph theory. Up till now, several graph operations such as the disjoint union, the Cartesian product, the Kronecker product, the corona, the edge corona, the neighborhood corona etc. have been introduced, and their spectra were computed in \cite{kn:Barik07,kn:Barik08,kn:Brouwer12,kn:Cui12,kn:Cvetkovic95,kn:Cvetkovic10,kn:Gopalapillai11,kn:Hou10,kn:McLeman11,kn:Nath14,kn:Wang12,kn:LiuZhou13,kn:LiuLu13,kn:Lan14}, respectively. Since cospectral
graphs have the same characteristic polynomials, graph operations whose corresponding spectra are known can be used to construct infinitely many pairs of cospectral graphs\cite{kn:McLeman11,kn:Cui12,kn:Wang12,kn:LiuZhou13,kn:LiuLu13,kn:Lan14}. On the other hand, comparing the exponents
and coefficients of two characteristic polynomials is a frequently-used and efficient method for determining the non-cospectrality of two graphs \cite{kn:Haemers08,kn:LiuWang11,kn:Liu14,kn:ZhangLiu09}. Moreover, the spectra of graphs as well as the characteristic polynomials of graphs can help us to investigate many other properties of graphs, such as the \emph{energy}\cite{kn:Ilic09,kn:Ilic11,kn:Li12,kn:LiuZhou12}, the number of \emph{spanning trees}\cite{kn:Atajan06,kn:Ozeki11,kn:ZhangYong00,kn:ZhangYong05}, the \emph{Kirchhoff index}\cite{kn:Gao12,kn:WangWZ13,kn:Xu03,kn:ZhangHP09} and so on.

The \emph{subdivision graph}\cite{kn:Cvetkovic10} $\mathcal{S}(G)$ of a graph $G$ is the graph obtained by inserting a new vertex into every edge of $G$. We denote the set of such new vertices by $I(G)$.  In \cite{kn:Indulal12}, the following graph operations based on subdivision graphs were introduced.

\begin{defn}
\label{defn1}
{\em The \emph{subdivision-vertex join} of two vertex disjoint graphs $G_1$ and $G_2$, denoted by $G_1\dot{\vee}G_2$, is the graph obtained from $\mathcal{S}(G_1)$ and $G_2$ by joining each vertex of $V(G_1)$ with every vertex of $V(G_2)$.}
\end{defn}

\begin{defn}
\label{defn2}
{\em The \emph{subdivision-edge join} of two vertex disjoint graphs $G_1$ and $G_2$, denoted by $G_1\underline{\vee}G_2$, is the graph obtained from $\mathcal{S}(G_1)$ and $G_2$ by joining each vertex of $I(G_1)$ with every vertex of $V(G_2)$.}
\end{defn}

In  \cite{kn:Indulal12}, the $A$-spectra of $G_1\dot{\vee}G_2$ (respectively, $G_1\underline{\vee}G_2$), when $G_1$ and $G_2$ are both regular graphs, were computed in terms of the $A$-spectra of $G_1$ and $G_2$. As an application, the author constructed infinite family of new integral graphs (graphs with their spectra only integers).

In this paper, we determine the $A$-spectra of $G_1\dot{\vee} G_2$ (respectively, $G_1\underline{\vee} G_2$) for a regular graph $G_1$ and an arbitrary graph $G_2$ in terms of that of $G_1$ and $G_2$ (see Theorems \ref{SVAthm1} and \ref{SEAthm1}); this generalises \cite[Theorems 1.1 and 1.2]{kn:Indulal12}. We also determine the $L$-spectra and the $Q$-spectra of $G_1\dot{\vee} G_2$ (respectively, $G_1\underline{\vee} G_2$) for a regular graph $G_1$ and an arbitrary graph $G_2$ (see Theorems \ref{SVLthm1}, \ref{SVQthm1}, \ref{SELthm1} and \ref{SEQthm1}). As applications, our results on the spectra of $G_1\dot{\vee} G_2$ and $G_1\underline{\vee} G_2$ enable us to construct infinitely many pairs of cospectral graphs and help us to obtain the number of spanning trees and the Kirchhoff index of $G_1\dot{\vee}G_2$ and $G_1\underline{\vee}G_2$, respectively.

\section{Spectra of subdivision-vertex joins}\label{SV:spec}

Let $G$ be a graph on $n$ vertices and $m$ edges. The \emph{incidence matrix} $R(G)$ of $G$ is the $n\times m$ $(0,1)$-matrix $(b_{ij})$ such that $b_{ij}=1$ if and only if the vertex $v_i$ and edge $e_j$ are incident in $G$. The \emph{line graph} $\mathcal{L}(G)$ of a graph $G$ is the graph with vertices the edges of $G$ such that two vertices are adjacent if and only if the corresponding edges have a common end-vertex. It is well known \cite{kn:Cvetkovic10} that
\begin{equation}\label{RTRNN}
R(G)^TR(G)=A(\mathcal {L}(G))+2I_{m},
\end{equation}
where $R(G)^T$ means the transpose of $R(G)$. In particular, if $G$ is an $r$-regular graph, then
\begin{equation}\label{RRINie}
 R(G)R(G)^T=A(G)+rI_{n}.
\end{equation}

The \emph{$M$-coronal} $\Gamma_M(x)$ of an $n\times n$ matrix $M$ is defined \cite{kn:McLeman11,kn:Cui12} to be the sum of the entries of the matrix $(x I_n-M)^{-1}$, that is,
$$\Gamma_M(x)=\mathbf{1}_n^T(x I_n-M)^{-1}\mathbf{1}_n,$$
where $\mathbf{1}_n$ denotes the column vector of dimension $n$ with all the components equal to one. It is known \cite[Proposition 2]{kn:Cui12} that, if $M$ is an $n \times n$ matrix with each row sum equal to a constant $t$, then
\be
\label{eq:GammaT}
\Gamma_{M}(x) = n/(x-t).
\ee
In particular, since for any graph $G$ with $n$ vertices, each row sum of $L(G)$ is equal to $0$, we have
\be
\label{eq:GammaTL}
\Gamma_{L(G)}(x) = n/x.
\ee

Before proceeding to present the main results of this section, we need to state some results which will be used frequently later.

\begin{prop}\emph{\cite{kn:Schur}}\label{schur1212}
Let $M_1$, $M_2$, $M_3$ and $M_4$ be respectively $p\times p$, $p\times q$, $q\times p$ and $q\times q$ matrices with $M_1$ and $M_4$ invertible. Then
\begin{eqnarray*}
  \det\pmat{M_1 & M_2\\
            M_3 & M_4} &=& \det(M_4)\cdot\det\left(M_1-M_2M_4^{-1}M_3\right), \label{schur1}   \\
   &=& \det(M_1)\cdot\det\left(M_4-M_3M_1^{-1}M_2\right) \label{schur2},
\end{eqnarray*}
where $M_1-M_2M_4^{-1}M_3$ and $M_4-M_3M_1^{-1}M_2$ are called the Schur complements of $M_4$ and $M_1$, respectively.
\end{prop}

\begin{prop}\label{DETEXP1}
Let $A$ be an $n\times n$ real matrix, and $J_{s\times t}$ denote the $s\times t$ matrix with all entries equal to one. Then
\[\det(A+\alpha J_{n\times n})=\det(A)+\alpha\mathbf{1}_n^T\adj(A)\mathbf{1}_n,\]
where $\alpha$ is an real number and $\adj(A)$ is the adjugate matrix of $A$.
\end{prop}

\begin{proof}
Let $A=\left(\begin{array}{cccc}
               \mathbf{a}_1 \\
               \mathbf{a}_2 \\
               \vdots        \\
               \mathbf{a}_n
             \end{array}
\right)$, where $\mathbf{a}_i=\left(a_{i,1},a_{i,2},\ldots,a_{i,n}\right)$ and $a_{i,j}$ is $(i,j)$-entry of $A$. Then
\begin{eqnarray*}
\det(A+\alpha J_{n\times n})&=& \det\left(\begin{array}{c}
               \mathbf{a}_1 \\
               \left(\begin{array}{c}
               \mathbf{a}_2 \\
               \vdots        \\
               \mathbf{a}_n
             \end{array}\right)+\alpha J_{(n-1)\times n}\end{array}\right)+\det\left(\begin{array}{c}
               \alpha\mathbf{1}_n^T \\
               \left(\begin{array}{c}
               \mathbf{a}_2 \\
               \vdots        \\
               \mathbf{a}_n
             \end{array}\right)+\alpha J_{(n-1)\times n}\end{array}\right)\\
&=& \det\left(\begin{array}{c}
               \mathbf{a}_1 \\
               \left(\begin{array}{c}
               \mathbf{a}_2 \\
               \vdots        \\
               \mathbf{a}_n
             \end{array}\right)+\alpha J_{(n-1)\times n}\end{array}\right)+\det\left(\begin{array}{c}
               \alpha\mathbf{1}_n^T \\
               \mathbf{a}_2 \\
               \vdots        \\
               \mathbf{a}_n
             \end{array}\right)\\
&=& \det\left(\begin{array}{c}
               \mathbf{a}_1 \\
               \left(\begin{array}{c}
               \mathbf{a}_2 \\
               \vdots        \\
               \mathbf{a}_n
             \end{array}\right)+\alpha J_{(n-1)\times n}\end{array}\right)+\alpha\sum_{i=1}^nC_{1,i} \\
&=& \det\left(\begin{array}{c}
               \mathbf{a}_1 \\
               \mathbf{a}_2 \\
               \left(\begin{array}{c}
               \mathbf{a}_3 \\
               \vdots        \\
               \mathbf{a}_n
             \end{array}\right)+\alpha J_{(n-2)\times n}\end{array}\right)+\det\left(\begin{array}{c}
               \mathbf{a}_1  \\
               \alpha\mathbf{1}_n^T \\
               \left(\begin{array}{c}
               \mathbf{a}_3 \\
               \vdots        \\
               \mathbf{a}_n
             \end{array}\right)+\alpha J_{(n-2)\times n}\end{array}\right)+\alpha\sum_{j=1}^nC_{1,j}\\
&=& \det\left(\begin{array}{c}
               \mathbf{a}_1 \\
               \mathbf{a}_2 \\
               \left(\begin{array}{c}
               \mathbf{a}_3 \\
               \vdots        \\
               \mathbf{a}_n
             \end{array}\right)+\alpha J_{(n-2)\times n}\end{array}\right)+\det\left(\begin{array}{c}
               \mathbf{a}_1 \\
               \alpha\mathbf{1}_n^T \\
               \mathbf{a}_3   \\
               \vdots        \\
               \mathbf{a}_n
             \end{array}\right)+\alpha\sum_{j=1}^nC_{1,j}\\
&=& \det\left(\begin{array}{c}
               \mathbf{a}_1 \\
               \mathbf{a}_2 \\
               \left(\begin{array}{c}
               \mathbf{a}_3 \\
               \vdots        \\
               \mathbf{a}_n
             \end{array}\right)+\alpha J_{(n-2)\times n}\end{array}\right)+\alpha\sum_{j=1}^nC_{2,j}+\alpha\sum_{j=1}^nC_{1,j}\\
&=& \cdots \\
&=&\det(A)+\alpha\sum_{i=1}^n\sum_{j=1}^nC_{i,j}\\
&=&\det(A)+\alpha\mathbf{1}_n^T\adj(A)\mathbf{1}_n,
\end{eqnarray*}
where $C_{i,j}$ is the cofactor of the $(i,j)$-entry of $\det(A)$.
\qed\end{proof}

\begin{cor}\label{JnnExpand}
Let $A$ be an $n\times n$ real matrix. Then
\[\det(xI_n-A-\alpha J_{n\times n})=\left(1-\alpha\Gamma_A(x)\right)\det(xI_n-A).\]
\end{cor}
\begin{proof}
Note that $\adj(xI_n-A)=\det(xI_n-A)\cdot(xI_n-A)^{-1}$. Then by Proposition \ref{DETEXP1}, we have
\begin{eqnarray*}
  \det(xI_n-A-\alpha J_{n\times n})&=&\det(xI_n-A)-\alpha\mathbf{1}_n^T\adj(xI_n-A)\mathbf{1}_n \\
&=& \det(xI_n-A)\left(1-\alpha\mathbf{1}_n^T(xI_n-A)^{-1}\mathbf{1}_n\right) \\
&=&\left(1-\alpha\Gamma_A(x)\right)\det(xI_n-A).
\end{eqnarray*}
The required result is obtained.
\qed\end{proof}

\subsection{$A$-spectra of subdivision-vertex joins}
\label{sec:SVA}

\begin{thm}\label{SVAthm1}
Let $G_1$ be an $r_1$-regular graph on $n_1$ vertices and $m_1$ edges, and $G_2$ an arbitrary graph on $n_2$ vertices. Then
\begin{eqnarray*}
\phi\left(A(G_1\dot{\vee} G_2);x\right)=\phi(A(G_2);x)\cdot x^{m_1-n_1}\cdot\Big(x^2-n_1 x\Gamma_{A(G_2)}(x)-2r_1\Big)\cdot\prod_{i=2}^{n_1}\Big(x^2-r_1-\lambda_i(G_1)\Big).
\end{eqnarray*}
\end{thm}
\begin{proof}
Let $R$ be the incidence matrix of $G_1$. Then, with a proper labeling of vertices, the adjacency matrix of $G_1\dot{\vee} G_2$ can be written as
\[A(G_1\dot{\vee} G_2)=\bmat{
                          0_{n_1\times n_1} & R & J_{n_1\times n_2} \\[0.2cm]
                          R^T & 0_{m_1\times m_1} & 0_{m_1\times n_2}\\[0.2cm]
                          J_{n_2\times n_1} & 0_{n_2\times m_1} & A(G_2)
                       },\]
where $J_{s\times t}$ denotes the $s\times t$ matrix with all entries equal to one, and $0_{s\times t}$ denotes the $s\times t$ matrix with all entries equal to zero. Then the adjacency characteristic polynomial of $G_1\dot{\vee} G_2$ is given by
\begin{eqnarray*}
\phi\left(A(G_1\dot{\vee} G_2);x\right)
&=& \det\bmat{
                          xI_{n_1} & -R & -J_{n_1\times n_2} \\[0.2cm]
                          -R^T & xI_{m_1} & 0_{m_1\times n_2}\\[0.2cm]
                          -J_{n_2\times n_1} & 0_{n_2\times m_1} & xI_{n_2}-A(G_2)
                       }\\ [0.2cm]
&=& \det(xI_{n_2}-A(G_2))\cdot\det(S)\\ [0.2cm]
&=& \phi(A(G_2))\cdot\det(S),
\end{eqnarray*}
where
\begin{eqnarray*}
S&=&\pmat{
          xI_{n_1} & -R \\[0.2cm]
         -R^T & xI_{m_1}
       }-\pmat{J_{n_1\times n_2}\\[0.2cm]
                0_{m_1\times n_2}}(xI_{n_2}-A(G_2))^{-1}\pmat{J_{n_2\times n_1} & 0_{n_2\times m_1}}\\ [0.2cm]
&=&\pmat{
          xI_{n_1}-\Gamma_{A(G_2)}(x)J_{n_1\times n_1} & -R \\[0.2cm]
         -R^T & xI_{m_1}
       }
\end{eqnarray*}
is the Schur complement of $xI_{n_2}-A(G_2)$ obtained by Proposition \ref{schur1212}. By Proposition \ref{schur1212} and Corollary \ref{JnnExpand}, we have
\begin{eqnarray*}
\det (S)&=&x^{m_1}\cdot\det\left(xI_{n_1}-\Gamma_{A(G_2)}(x)J_{n_1\times n_1}-\frac{1}{x}RR^T\right) \\
&=& x^{m_1}\cdot\left(1-\Gamma_{A(G_2)}(x)\cdot\Gamma_{\frac{1}{x}RR^T}(x)\right)\cdot\det\left(xI_{n_1}-\frac{1}{x}RR^T\right).
\end{eqnarray*}
By (\ref{RRINie}) and (\ref{eq:GammaT}), we have
\begin{equation}\label{RRTxGa111}
 \Gamma_{\frac{1}{x}RR^T}(x)=\frac{n_1}{x-\frac{2r_1}{x}}=\frac{n_1x}{x^2-2r_1}.
\end{equation}
By plugging (\ref{RRTxGa111}) into $\det (S)$, and then applying the fact that $f(\lambda)$ is an eigenvalue of $f(A)$ if $\lambda$ is an eigenvalue of a matrix $A$ and $f(A)$ is a polynomial of $A$, we have
\begin{eqnarray*}
\det (S)&=& x^{m_1}\cdot\left(1-\frac{n_1x}{x^2-2r_1}\cdot\Gamma_{A(G_2)}(x)\right)\cdot\prod_{i=1}^{n_1}\left(x-\frac{r_1}{x}-\frac{1}{x}\lambda_i(G_1)\right) \\
&=& x^{m_1-n_1}\cdot\Big(x^2-2r_1-n_1x\cdot\Gamma_{A(G_2)}(x)\Big)\cdot\prod_{i=2}^{n_1}\Big(x^2-r_1-\lambda_i(G_1)\Big).
\end{eqnarray*}
Here in the last step we used the fact that $\lambda_1(G_1)=r_1$ given $G_1$ is an $r_1$-regular graph. Then the required result follows from $\phi\left(A(G_1\dot{\vee} G_2);x\right)=\phi(A(G_2))\cdot\det(S)$.  \qed\end{proof}

Theorem \ref{SVAthm1} enables us to compute the $A$-spectra of many subdivision-vertex joins. In general, if we can determine the $A(G_2)$-coronal $\Gamma_{A(G_2)}(x)$, then we are able to compute the $A$-spectrum of $G_1\dot{\vee} G_2$. Let $K_{p,q}$ denote the complete bipartite graph with $p, q \ge 1$ vertices in the two parts of its bipartition. Then \cite{kn:McLeman11} the $A(K_{p,q})$-coronal of $K_{p,q}$ is given by
\begin{equation}\label{CornCPL}
\Gamma_{A(K_{p,q})}(x) = ((p+q)x+2pq)/(x^2-pq).
\end{equation}
Note \cite{kn:McLeman11,kn:Cui12} that if $G$ is an $r$-regular graph on $n$ vertices, then the $A(G)$-coronal of $G$ is given by
\begin{equation}\label{CornRRL}
\Gamma_{A(G)}(x)=n/(x-r).
\end{equation}
By substituting (\ref{CornCPL}) and (\ref{CornRRL}) back into Theorem \ref{SVAthm1},  we have the following corollaries. The computations are routine, and hence we omit the details.

\begin{cor}\label{SVAComplete}
Let $G$ be an $r$-regular graph on $n$ vertices and $m$ edges with $m\ge n$. Then the $A$-spectrum of $G  \dot{\vee} K_{p,q}$ consists of:
\begin{itemize}
  \item[\rm (a)] $0$, repeated $m-n+p+q-2$ times;
  \item[\rm (b)] $\pm\sqrt{r+\lambda_i(G)}$, for $i=2,3,\ldots,n$;
  \item[\rm (c)] four roots of the equation $x^4-\big(pq+n(p+q)+2r\big)x^2-2npqx+2pqr=0.$
\end{itemize}
\end{cor}

\begin{cor}\label{SVACor1}
\emph{\cite[Theorem 1.1]{kn:Indulal12}}
Let $G_1$ be an $r_1$-regular graphs on $n_1$ vertices and $m_1$ edges, and $G_2$ an $r_2$-regular graphs on $n_2$ vertices. Then the $A$-spectrum of $G_1\dot{\vee} G_2$ consists of:
\begin{itemize}
  \item[\rm (a)] $\lambda_i(G_2)$, for $i=2,3,\ldots,n_2$;
  \item[\rm (b)] $0$, repeated $m_1-n_1$ times;
  \item[\rm (c)] $\pm\sqrt{r_1+\lambda_j(G_1)}$, for $j=2,3,\ldots,n_1$;
  \item[\rm (d)] three roots of the equation $x^3-r_2x^2-(n_1n_2+2r_1)x+2r_1r_2=0.$
\end{itemize}
\end{cor}

Up till now, many infinite families of pairs of $A$-cospectral graphs are generated by using graph operations (for example, \cite{kn:Barik07,kn:LiuLu13,kn:LiuZhou13,kn:McLeman11}). Here, we use the subdivision-vertex join to construct infinitely many pairs of $A$-cospectral graphs, as stated in the following corollary of Theorem \ref{SVAthm1}.

\begin{cor}
\label{cor:SVAcosp}
\begin{itemize}
  \item[\rm (a)]  If $G_1$ and $G_2$ are $A$-cospectral regular graphs, and $H$ is any graph, then $G_1\dot{\vee} H$ and $G_2\dot{\vee} H$ are $A$-cospectral.
  \item[\rm (b)]  If $G$ is a regular graph, and $H_1$ and $H_2$ are $A$-cospectral graphs with $\Gamma_{A(H_1)}(x)=\Gamma_{A(H_2)}(x)$, then $G\dot{\vee} H_1$ and $G\dot{\vee} H_2$ are $A$-cospectral.
\end{itemize}
\end{cor}

Note that the condition $\Gamma_{A(H_1)}(x)=\Gamma_{A(H_2)}(x)$ in (b) is not redundant since $A$-cospectral graphs may have different $A$-coronals \cite[Remark 3]{kn:McLeman11}.

\subsection{$L$-spectra of subdivision-vertex joins}
\label{sec:SVL}

\begin{thm}\label{SVLthm1}
Let $G_1$ be an $r_1$-regular graph on $n_1$ vertices and $m_1$ edges, and $G_2$ an arbitrary graph on $n_2$ vertices. Then
\begin{eqnarray*}
\phi\left(L(G_1\dot{\vee} G_2);x\right)&=&x\cdot(x-2)^{m_1-n_1}\cdot\Big(x^2-(2+r_1+n_1+n_2)x+2n_1+2n_2+n_1r_1\Big)\\
&&\cdot\prod_{i=2}^{n_2}\Big(x-n_1-\mu_i(G_2)\Big)\cdot\prod_{i=2}^{n_1}\Big(x^2-\big(2+r_1+n_2\big)x+2n_2+\mu_i(G_1)\Big).
\end{eqnarray*}
\end{thm}
\begin{proof}
Let $R$ be the incidence matrix of $G_1$. Then the Laplacian matrix of $G_1\dot{\vee} G_2$ can be written as
\[L(G_1\dot{\vee} G_2)=\bmat{
                          (r_1+n_2)I_{n_1} & -R & -J_{n_1\times n_2} \\[0.2cm]
                          -R^T & 2I_{m_1} & 0_{m_1\times n_2}\\[0.2cm]
                          -J_{n_2\times n_1} & 0_{n_2\times m_1} & n_1I_{n_2}+L(G_2)
                       }.\]
Thus the Laplacian characteristic polynomial of $G_1\dot{\vee} G_2$ is given by
\begin{eqnarray*}
\phi\left(L(G_1\dot{\vee} G_2);x\right)
&=& \det\bmat{
                          (x-r_1-n_2)I_{n_1} & R & J_{n_1\times n_2} \\[0.2cm]
                          R^T & (x-2)I_{m_1} & 0_{m_1\times n_2}\\[0.2cm]
                          J_{n_2\times n_1} & 0_{n_2\times m_1} & (x-n_1)I_{n_2}-L(G_2)
                       }\\ [0.2cm]
&=& \det\big((x-n_1)I_{n_2}-L(G_2)\big)\cdot\det(S)\\ [0.2cm]
&=& \det(S)\cdot\prod_{i=1}^{n_2}\Big(x-n_1-\mu_i(G_2)\Big),
\end{eqnarray*}
where
\begin{eqnarray*}
S&=&\pmat{
          (x-r_1-n_2)I_{n_1}-\Gamma_{L(G_2)}(x-n_1)J_{n_1\times n_1} & -R \\[0.2cm]
         -R^T & (x-2)I_{m_1}
       }
\end{eqnarray*}
is the Schur complement of $(x-n_1)I_{n_2}-L(G_2)$ obtained by Proposition \ref{schur1212}. Then by Proposition \ref{schur1212} and Corollary \ref{JnnExpand}, we have
\begin{eqnarray*}
\det (S)&=&(x-2)^{m_1}\cdot\det\left((x-r_1-n_2)I_{n_1}-\Gamma_{L(G_2)}(x-n_1)J_{n_1\times n_1}-\frac{1}{x-2}RR^T\right) \\
&=& (x-2)^{m_1}\cdot\det\left((x-r_1-n_2)I_{n_1}-\frac{1}{x-2}RR^T\right)\\
&&\cdot\left(1-\Gamma_{L(G_2)}(x-n_1)\cdot\Gamma_{\frac{1}{x-2}RR^T}(x-r_1-n_2)\right).
\end{eqnarray*}
By (\ref{RRINie}) and (\ref{eq:GammaT}), we have
\begin{equation}\label{RRTx12}
\Gamma_{\frac{1}{x-2}RR^T}(x-r_1-n_2)=\frac{n_1}{x-r_1-n_2-\frac{2r_1}{x-2}}=\frac{n_1(x-2)}{x^2-(2+r_1+n_2)x+2n_2}.
\end{equation}
Then plugging (\ref{RRINie}), (\ref{eq:GammaTL}) and (\ref{RRTx12}) into $\det (S)$, we have
\begin{eqnarray*}
\det (S)&=& (x-2)^{m_1}\cdot\prod_{i=1}^{n_1}\left(x-r_1-n_2-\frac{r_1}{x-2}-\frac{1}{x-2}\lambda_i(G_1)\right)\\
&&\cdot\left(1- \frac{n_2}{x-n_1}\cdot\frac{n_1(x-2)}{x^2-(2+r_1+n_2)x+2n_2}\right)\\
&=&\frac{x}{x-n_1}\cdot(x-2)^{m_1-n_1}\cdot\prod_{i=2}^{n_1}\Big(x^2-\big(2+r_1+n_2\big)x+2n_2+\mu_i(G_1)\Big)\\
&&\cdot\Big(x^2-(2+r_1+n_1+n_2)x+2n_1+2n_2+n_1r_1\Big).
\end{eqnarray*}
Here the last step was obtained by applying the facts that $\mu_1(G_1)=0$ and $\lambda_i(G_1)=r_1-\mu_i(G_1)$ for $i=1,2,\ldots,n_1$. Then the required result follows from $\phi\left(L(G_1\dot{\vee} G_2);x\right)=\det(S)\cdot\prod\limits_{i=1}^{n_2}\Big(x-n_1-\mu_i(G_2)\Big)$ and the fact that $\mu_1(G_2)=0$.
\qed\end{proof}

Let $t(G)$ denote the number of spanning trees of $G$. It is well known \cite{kn:Cvetkovic95} that if $G$ is a connected graph on $n$ vertices with Laplacian spectrum $0=\mu_1(G)<\mu_{2}(G)\le\cdots\le\mu_n(G)$, then $$t(G)=\frac{\mu_{2}(G)\cdots\mu_n(G)}{n}.$$
Recently, computing the number of spanning trees of some specific types of graphs, especially the circulant graphs\cite{kn:Atajan06,kn:Ozeki11,kn:ZhangYong00,kn:ZhangYong05}, attracted many researchers' attention. Here, we obtain the number of the spanning trees of $G_1\dot{\vee} G_2$ for an $r_1$-regular graph $G_1$ and an arbitrary graph $G_2$.
\begin{cor}
\label{cor:Sptree}
Let $G_1$ be an $r_1$-regular graph on $n_1$ vertices and $m_1$ edges, and $G_2$ an arbitrary graph on $n_2$ vertices. Then
\begin{eqnarray*}
t(G_1\dot{\vee} G_2)=\frac{2^{m_1-n_1}\cdot(2n_1+2n_2+n_1r_1)\cdot\prod_{i=2}^{n_2}\Big(n_1+\mu_i(G_2)\Big)\cdot\prod_{i=2}^{n_1}\Big(2n_2+\mu_i(G_1)\Big)}{m_1+n_1+n_2}.
\end{eqnarray*}
\end{cor}
\begin{proof}
By Theorem \ref{SVLthm1}, the roots of $\phi\left(L(G_1\dot{\vee} G_2);x\right)$ are as follows:
\begin{itemize}
  \item[\rm (a)] $0$;
  \item[\rm (b)] $2$, repeated $m_1-n_1$ times;
  \item[\rm (c)] $n_1+\mu_i(G_2)$, for $i=2,3,\ldots,n_2$;
  \item[\rm (d)] two roots $x_1, x_2$ of the equation $x^2-(2+r_1+n_1+n_2)x+2n_1+2n_2+n_1r_1$;
  \item[\rm (e)] two roots $x_{j,1}, x_{j,2}$ of the equation $x^2-\big(2+r_1+n_2\big)x+2n_2+\mu_j(G_1)$, for $j=2,3,\ldots,n_1$.
\end{itemize}
By the relation betwen coefficients and roots of a polynomial, we have $x_1x_2=2n_1+2n_2+n_1r_1$, and $x_{j,1}x_{j,2}=2n_2+\mu_j(G_1)$ for $j=2,3,\ldots,n_1$. Then the required result is obtained by the definition of the number of spanning trees of a graph.
\qed\end{proof}

The \emph{Kirchhoff index} of a graph $G$, denoted by $Kf(G)$, is defined as the sum of resistance distances between all pairs of vertices \cite{kn:Bonchev94,kn:Klein93}. At almost exactly the same time, Gutman et al. \cite{kn:Gutman96} and Zhu et al. \cite{kn:Zhu96} proved that the Kirchhoff index of a connected graph $G$ with $n\,(n\geq2)$ vertices can be expressed as $$Kf(G)= n\sum_{i=2}^{n}\frac{1}{\mu_i(G)},$$
where $\mu_2(G), \ldots, \mu_n(G)$ are the non-zero Laplacian eigenvalues of $G$. Up till now, the Kirchhoff indices of many graph operations have been investigated, such as products, lexicographic products, joins,
coronae, clusters, line (respectively, subdivision and total) graphs of regular graphs and so on \cite{kn:Gao12,kn:WangWZ13,kn:Xu03,kn:ZhangHP09}. By Theorem \ref{SVLthm1}, we obtain the Kirchhoff index of the subdivision-vetex join $G_1\dot{\vee} G_2$ for an $r_1$-regular graph $G_1$ and an arbitrary graph $G_2$.

\begin{cor}
\label{cor:KirIn}
Let $G_1$ be an $r_1$-regular graph on $n_1$ vertices and $m_1$ edges, and $G_2$ an arbitrary graph on $n_2$ vertices. Then
\begin{align*}
   Kf(G_1\dot{\vee} G_2) &  = (m_1+n_1+n_2)\\
                         &  \quad\times\left(\frac{m_1-n_1}{2}+\frac{2+r_1+n_1+n_2}{2n_1+2n_2+n_1r_1}+\sum_{i=2}^{n_2}\frac{1}{n_1+\mu_i(G_2)}+\sum_{i=2}^{n_1}\frac{2+r_1+n_2}{2n_2+\mu_i(G_1)}\right).
\end{align*}
\end{cor}
\begin{proof} Notice that the roots of $\phi\left(L(G_1\dot{\vee} G_2);x\right)$ are given in the proof of Corollary \ref{cor:Sptree}. By the relation betwen coefficients and roots of a polynomial, we have $x_1+x_2=2+r_1+n_1+n_2$, and $x_{j,1}+x_{j,2}=2+r_1+n_2$ for $j=2,3,\ldots,n_1$. Then
\[\dfrac{1}{x_1}+\dfrac{1}{x_2}=\dfrac{x_1+x_2}{x_1x_2}=\dfrac{2+r_1+n_1+n_2}{2n_1+2n_2+n_1r_1},\]
and
\[\dfrac{1}{x_{j,1}}+\dfrac{1}{x_{j,2}}=\dfrac{x_{j,1}+x_{j,2}}{x_{j,1}x_{j,2}}=\dfrac{2+r_1+n_2}{2n_2+\mu_j(G_1)},\] for $j=2,3,\ldots,n_1$. So the required result is obtained by the definition of the Kirchhoff index of a graph.
\qed\end{proof}

Similar to Corollary \ref{cor:SVAcosp}, Theorem \ref{SVLthm1} enables us to construct infinitely many pairs of $L$-cospectral graphs.

\begin{cor}
\label{cor:SVLcosp}
\begin{itemize}
  \item[\rm (a)]  If $G_1$ and $G_2$ are $L$-cospectral regular graphs, and $H$ is an arbitrary graph, then $G_1\dot{\vee} H$ and $G_2\dot{\vee} H$ are $L$-cospectral.
  \item[\rm (b)]  If $G$ is a regular graph, and $H_1$ and $H_2$ are $L$-cospectral graphs, then $G\dot{\vee} H_1$ and $G\dot{\vee} H_2$ are $L$-cospectral.
  \item[\rm (c)]  If $G_1$ and $G_2$ are $L$-cospectral regular graphs, and $H_1$ and $H_2$ are $L$-cospectral graphs, then $G_1\dot{\vee} H_1$ and $G_2\dot{\vee} H_2$ are $L$-cospectral.
\end{itemize}
\end{cor}

\subsection{$Q$-spectra of subdivision-vertex joins}
\label{sec:SVQ}

\begin{thm}\label{SVQthm1}
Let $G_1$ be an $r_1$-regular graph on $n_1$ vertices and $m_1$ edges, and $G_2$ an arbitrary graph on $n_2$ vertices. Then
\begin{eqnarray*}
\phi\left(Q(G_1\dot{\vee} G_2);x\right)&=&(x-2)^{m_1-n_1}\cdot\Big(x^2-\big(2+r_1+n_2\big)x+2n_2-n_1(x-2)\cdot\Gamma_{Q(G_2)}(x-n_1)\Big)\\
&&\cdot\prod_{i=1}^{n_2}\Big(x-n_1-\nu_i(G_2)\Big)\cdot\prod_{i=1}^{n_1-1}\Big(x^2-\big(2+r_1+n_2\big)x+2(r_1+n_2)-\nu_i(G_1)\Big).
\end{eqnarray*}
\end{thm}
\begin{proof}
Let $R$ be the incidence matrix of $G_1$. Then the signless Laplacian matrix of $G_1\dot{\vee} G_2$ can be written as
\[Q(G_1\dot{\vee} G_2)=\bmat{
                          (r_1+n_2)I_{n_1} & R & J_{n_1\times n_2} \\[0.2cm]
                          R^T & 2I_{m_1} & 0_{m_1\times n_2}\\[0.2cm]
                          J_{n_2\times n_1} & 0_{n_2\times m_1} & n_1I_{n_2}+Q(G_2)
                       }.\]
The result follows by applying $RR^T=Q(G_1)$ and refining the arguments used to prove Theorem \ref{SVLthm1}.
\qed\end{proof}

Again, by applying (\ref{eq:GammaT}), Theorem \ref{SVQthm1} implies the following result.
\begin{cor}\label{SVQcor1}
Let $G_1$ be an $r_1$-regular graphs on $n_1$ vertices and $m_1$ edges, and $G_2$ an $r_2$-regular graphs on $n_2$ vertices. Then
\begin{eqnarray*}
\phi\left(Q(G_1\dot{\vee} G_2);x\right)&=&(x-2)^{m_1-n_1}\cdot(x^3-ax^2+bx-4r_2n_2)\cdot\prod_{i=1}^{n_2-1}\Big(x-n_1-\nu_i(G_2)\Big)\\
&&\cdot\prod_{i=1}^{n_1-1}\Big(x^2-\big(2+r_1+n_2\big)x+2(r_1+n_2)-\nu_i(G_1)\Big),
\end{eqnarray*}
where $a=2+2r_2+r_1+n_1+n_2$ and $b=2n_1+2n_2+n_1r_1+2r_1r_2+2r_2n_2+4r_2$.
\end{cor}

Theorem \ref{SVQthm1} enables us to construct infinitely many pairs of $Q$-cospectral graphs.

\begin{cor}
\label{cor:SVQcosp}
\begin{itemize}
  \item[\rm (a)]  If $G_1$ and $G_2$ are $Q$-cospectral regular graphs, and $H$ is a regular graph, then $G_1\dot{\vee} H$ and $G_2\dot{\vee} H$ are $Q$-cospectral.
  \item[\rm (b)]  If $G$ is a regular graph, and $H_1$ and $H_2$ are $Q$-cospectral graphs with $\Gamma_{Q(H_1)}(x)=\Gamma_{Q(H_2)}(x)$, then $G\dot{\vee} H_1$ and $G\dot{\vee} H_2$ are $Q$-cospectral.
\end{itemize}
\end{cor}
Similar to Corollary \ref{cor:SVAcosp}, the condition $\Gamma_{Q(H_1)}(x)=\Gamma_{Q(H_2)}(x)$ in (b) is not redundant since $Q$-cospectral graphs may have different $Q$-coronals.

\section{Spectra of subdivision-edge joins}\label{SE:spec}
In this section, we determine the spectra of subdivision-edge joins.
\subsection{$A$-spectra of subdivision-edge joins}
\label{sec:SEA}

\begin{thm}\label{SEAthm1}
Let $G_1$ be an $r_1$-regular graph on $n_1$ vertices and $m_1$ edges, and $G_2$ an arbitrary graph on $n_2$ vertices. Then
\begin{eqnarray*}
\phi\left(A(G_1 \underline{\vee}  G_2);x\right)=\phi(A(G_2);x)\cdot x^{m_1-n_1}\cdot\Big(x^2-m_1 x\Gamma_{A(G_2)}(x)-2r_1\Big)\cdot\prod_{i=2}^{n_1}\Big(x^2-r_1-\lambda_i(G_1)\Big).
\end{eqnarray*}
\end{thm}
\begin{proof}
Let $R$ be the incidence matrix of $G_1$. Then the adjacency matrix of $G_1\underline{\vee} G_2$ can be written as
\[A(G_1\underline{\vee} G_2)=\bmat{
                          0_{n_1\times n_1} & R & 0_{n_1\times n_2} \\[0.2cm]
                          R^T & 0_{m_1\times m_1} & J_{m_1\times n_2}\\[0.2cm]
                          0_{n_2\times n_1} & J_{n_2\times m_1} & A(G_2)
                       }.\]
Thus the adjacency characteristic polynomial of $G_1\underline{\vee} G_2$ is given by
\begin{eqnarray*}
\phi\left(A(G_1\underline{\vee} G_2);x\right)
&=& \det\bmat{
                          xI_{n_1} & -R & 0_{n_1\times n_2} \\[0.2cm]
                          -R^T & xI_{m_1} & -J_{m_1\times n_2}\\[0.2cm]
                          0_{n_2\times n_1} & -J_{n_2\times m_1} & xI_{n_2}-A(G_2)
                       }\\ [0.2cm]
&=& \det(xI_{n_2}-A(G_2))\cdot\det(S)\\ [0.2cm]
&=& \phi(A(G_2))\cdot\det(S),
\end{eqnarray*}
where
\begin{eqnarray*}
S&=&\pmat{
          xI_{n_1} & -R \\[0.2cm]
         -R^T & xI_{m_1}-\Gamma_{A(G_2)}(x)J_{m_1\times m_1}
       }
\end{eqnarray*}
is the Schur complement of $xI_{n_2}-A(G_2)$ obtained by Proposition \ref{schur1212}. By Proposition \ref{schur1212} and Corollary \ref{JnnExpand}, we have
\begin{eqnarray*}
\det (S)&=&x^{n_1}\cdot\det\left(xI_{m_1}-\Gamma_{A(G_2)}(x)J_{m_1\times m_1}-\frac{1}{x}R^TR\right) \\
&=& x^{n_1}\cdot\left(1-\Gamma_{A(G_2)}(x)\cdot\Gamma_{\frac{1}{x}R^TR}(x)\right)\cdot\det\left(xI_{m_1}-\frac{1}{x}R^TR\right).
\end{eqnarray*}
By (\ref{RTRNN}) and (\ref{eq:GammaT}), we have
\begin{equation}\label{RTRG121G}
\Gamma_{\frac{1}{x}R^TR}(x)=\frac{m_1}{x-\frac{2r_1}{x}}=\frac{m_1x}{x^2-2r_1}.
\end{equation}
Note that \cite[Theorem 2.4.1]{kn:Cvetkovic10} the $A$-spectrum of $\mathcal {L}(G_1)$ are $\lambda_i(G_1)+r_1-2$ for $i=1,2,\ldots,n_1$, and $-2$ repeated $m_1-n_1$ times. Then by (\ref{RTRNN}) and (\ref{RTRG121G}), we have
\begin{eqnarray*}
\det (S)&=& x^{n_1}\cdot\left(1-\Gamma_{A(G_2)}(x)\cdot\frac{m_1x}{x^2-2r_1}\right)\cdot\prod_{i=1}^{m_1}\left(x-\frac{2}{x}-\frac{1}{x}\lambda_i(\mathcal{L}(G_1))\right) \\
&=& x^{m_1-n_1}\cdot\Big(x^2-2r_1-m_1x\cdot\Gamma_{A(G_2)}(x)\Big)\cdot\prod_{i=2}^{n_1}\Big(x^2-r_1-\lambda_i(G_1)\Big).
\end{eqnarray*}
Here the last step used the fact that $\lambda_1(G_1)=r_1$ since $G_1$ is an $r_1$-regular graph. Then the required result follows from $\phi\left(A(G_1\underline{\vee} G_2);x\right)=\phi(A(G_2))\cdot\det(S)$.
\qed
\end{proof}

Similar to Corollaries \ref{SVAComplete}, \ref{SVACor1} and \ref{cor:SVAcosp}, Theorem \ref{SEAthm1} implies the following results.

\begin{cor}\label{SEAComplete}
Let $G$ be an $r$-regular graph on $n$ vertices and $m$ edges with $m\ge n$. Then the $A$-spectrum of $G  \underline{\vee} K_{p,q}$ consists of:
\begin{itemize}
  \item[\rm (a)] $0$, repeated $m-n+p+q-2$ times;
  \item[\rm (b)] $\pm\sqrt{r+\lambda_i(G)}$, for $i=2,3,\ldots,n$;
  \item[\rm (c)] four roots of the equation $x^4-\big(pq+m(p+q)+2r\big)x^2-2mpqx+2pqr=0.$
\end{itemize}
\end{cor}

\begin{cor}\label{SEACor1}
\emph{\cite[Theorem 1.2]{kn:Indulal12}}
Let $G_1$ be an $r_1$-regular graphs on $n_1$ vertices and $m_1$ edges, and $G_2$ an $r_2$-regular graphs on $n_2$ vertices. Then the $A$-spectrum of $G_1\underline{\vee} G_2$ consists of:
\begin{itemize}
  \item[\rm (a)] $\lambda_i(G_2)$, for $i=2,3,\ldots,n_2$;
  \item[\rm (b)] $0$, repeated $m_1-n_1$;
  \item[\rm (c)] $\pm\sqrt{r_1+\lambda_j(G_1)}$, for $j=2,3,\ldots,n_1$;
  \item[\rm (d)] three roots of the equation $x^3-r_2x^2-(m_1n_2+2r_1)x+2r_1r_2=0.$
\end{itemize}
\end{cor}

\begin{cor}
\label{cor:SEAcosp}
\begin{itemize}
  \item[\rm (a)]  If $G_1$ and $G_2$ are $A$-cospectral regular graphs, and $H$ is any graph, then $G_1\underline{\vee} H$ and $G_2\underline{\vee} H$ are $A$-cospectral.
  \item[\rm (b)]  If $G$ is a regular graph, and $H_1$ and $H_2$ are $A$-cospectral graphs with $\Gamma_{A(H_1)}(x)=\Gamma_{A(H_2)}(x)$, then $G\underline{\vee} H_1$ and $G\underline{\vee} H_2$ are $A$-cospectral.
\end{itemize}
\end{cor}

\subsection{$L$-spectra of subdivision-edge joins}
\label{sec:SEL}

\begin{thm}\label{SELthm1}
Let $G_1$ be an $r_1$-regular graph on $n_1$ vertices and $m_1$ edges, and $G_2$ an arbitrary graph on $n_2$ vertices. Then
\begin{eqnarray*}
\phi\left(L(G_1 \underline{\vee}  G_2);x\right)&=&x\cdot(x-2-n_2)^{m_1-n_1}\cdot\Big(x^2-(2+r_1+m_1+n_2)x+r_1n_2+r_1m_1+2m_1\Big)\\
&&\cdot\prod_{i=2}^{n_2}\Big(x-m_1-\mu_i(G_2)\Big)\cdot\prod_{i=2}^{n_1}\Big(x^2-(2+r_1+n_2)x+r_1n_2+\mu_i(G_1)\Big).
\end{eqnarray*}
\end{thm}

\begin{proof}
Let $R$ be the incidence matrix of $G_1$. Then the Laplacian matrix of $G_1\underline{\vee} G_2$ can be written as
\[L(G_1\underline{\vee} G_2)=\bmat{
                          r_1I_{n_1} & -R & 0_{n_1\times n_2} \\[0.2cm]
                          -R^T & (2+n_2)I_{m_1} & -J_{m_1\times n_2}\\[0.2cm]
                          0_{n_2\times n_1} & -J_{n_2\times m_1} & m_1I_{n_2}+L(G_2)
                       }.\]
Thus the Laplacian characteristic polynomial of $G_1\underline{\vee} G_2$ is given by
\begin{eqnarray*}
\phi\left(L(G_1\underline{\vee} G_2);x\right)
&=& \det\bmat{
                          (x-r_1)I_{n_1} & R & 0_{n_1\times n_2} \\[0.2cm]
                          R^T & (x-2-n_2)I_{m_1} & J_{m_1\times n_2}\\[0.2cm]
                          0_{n_2\times n_1} & J_{n_2\times m_1} & (x-m_1)I_{n_2}-L(G_2)
                       }\\ [0.2cm]
&=& \det\big((x-m_1)I_{n_2}-L(G_2)\big)\cdot\det(S)\\ [0.2cm]
&=& \det(S)\cdot\prod_{i=1}^{n_2}\Big(x-m_1-\mu_i(G_2)\Big),
\end{eqnarray*}
where
\begin{eqnarray*}
S&=&\pmat{
          (x-r_1)I_{n_1} & R  \\[0.2cm]
          R^T & (x-2-n_2)I_{m_1}-\Gamma_{L(G_2)}(x-m_1)J_{m_1\times m_1}
       }
\end{eqnarray*}
is the Schur complement of $(x-m_1)I_{n_2}-L(G_2)$ obtained by Proposition \ref{schur1212}.  By Proposition \ref{schur1212} and Corollary \ref{JnnExpand}, we have
\begin{eqnarray*}
\det (S)&=&(x-r_1)^{n_1}\cdot\det\left((x-2-n_2)I_{m_1}-\Gamma_{L(G_2)}(x-m_1)J_{m_1\times m_1}-\frac{1}{x-r_1}R^TR\right) \\
&=& (x-r_1)^{n_1}\cdot\det\left((x-2-n_2)I_{m_1}-\frac{1}{x-r_1}R^TR\right) \\
&&\cdot\left(1-\Gamma_{L(G_2)}(x-m_1)\Gamma_{\frac{1}{x-r_1}R^TR}(x-2-n_2)\right) .
\end{eqnarray*}
By (\ref{RTRNN}) and (\ref{eq:GammaT}), we have
\begin{equation}\label{RTRxr121}
   \Gamma_{\frac{1}{x-r_1}R^TR}(x-2-n_2)=\frac{m_1}{x-2-n_2-\frac{2r_1}{x-r_1}}=\frac{m_1(x-r_1)}{x^2-(2+r_1+n_2)x+n_2r_1}.
\end{equation}
Then by (\ref{RTRNN}),  (\ref{eq:GammaTL}) and  (\ref{RTRxr121}), we have
\begin{eqnarray*}
\det (S)
&=& (x-r_1)^{n_1}\cdot\prod_{i=1}^{m_1}\left(x-2-n_2-\frac{2}{x-r_1}-\frac{1}{x-r_1}\lambda_i(\mathcal{L}(G_1))\right) \\
&&\cdot\left(1- \frac{n_2}{x-m_1}\cdot\frac{m_1(x-r_1)}{x^2-(2+r_1+n_2)x+n_2r_1}\right)\\
&=& \frac{x}{x-m_1}\cdot(x-2-n_2)^{m_1-n_1}\cdot\prod_{i=2}^{n_1}\Big(x^2-(2+r_1+n_2)x+r_1n_2+\mu_i(G_1)\Big)\\
&& \cdot\Big(x^2-(2+r_1+m_1+n_2)x+r_1n_2+r_1m_1+2m_1\Big).
\end{eqnarray*}
Here in the last step we used the facts that $\mu_1(G_1)=0$ and $\lambda_i(G_1)=r_1-\mu_i(G_1)$ for $i=1,2,\ldots,n_1$, and the fact that the $A$-spectrum of $\mathcal {L}(G_1)$ are $\lambda_i(G_1)+r_1-2$ for $i=1,2,\ldots,n_1$, and $-2$ repeated $m_1-n_1$ times. Then the required result follows from $\phi\left(L(G_1\underline{\vee} G_2);x\right)=\det(S)\cdot\prod\limits_{i=1}^{n_2}\Big(x-m_1-\mu_i(G_2)\Big)$ and the fact that $\mu_1(G_2)=0$.
\qed
\end{proof}

Similar to Corollaries \ref{cor:Sptree}, \ref{cor:KirIn} and \ref{cor:SVLcosp}, Theorem \ref{SELthm1} implies the following results.

\begin{cor}\label{SESptree}
Let $G_1$ be an $r_1$-regular graph on $n_1$ vertices and $m_1$ edges, and $G_2$ an arbitrary graph on $n_2$ vertices. Then
\begin{eqnarray*}
t(G_1 \underline{\vee}  G_2)=\frac{(2+n_2)^{m_1-n_1}\cdot(r_1n_2+r_1m_1+2m_1)\cdot\prod_{i=2}^{n_2}\Big(m_1+\mu_i(G_2)\Big)\cdot\prod_{i=2}^{n_1}\Big(r_1n_2+\mu_i(G_1)\Big)}{m_1+n_1+n_2}.
\end{eqnarray*}
\end{cor}

\begin{cor}\label{SEKirIn}
Let $G_1$ be an $r_1$-regular graph on $n_1$ vertices and $m_1$ edges, and $G_2$ an arbitrary graph on $n_2$ vertices. Then
\begin{align*}
  Kf(G_1 \underline{\vee}  G_2) &= (m_1+n_1+n_2)\\
                                &\quad\times\left(\frac{m_1-n_1}{2+n_2}+\frac{2+r_1+m_1+n_2}{r_1n_2+r_1m_1+2m_1}+\sum_{i=2}^{n_2}\frac{1}{m_1+\mu_i(G_2)}+\sum_{i=2}^{n_1}\frac{2+r_1+n_2}{r_1n_2+\mu_i(G_1)}\right).
\end{align*}
\end{cor}

\begin{cor}
\label{cor:SELcosp}
\begin{itemize}
  \item[\rm (a)]  If $G_1$ and $G_2$ are $L$-cospectral regular graphs, and $H$ is an arbitrary graph, then $G_1\underline{\vee} H$ and $G_2\underline{\vee} H$ are $L$-cospectral.
  \item[\rm (b)]  If $G$ is a regular graph, and $H_1$ and $H_2$ are $L$-cospectral graphs, then $G\underline{\vee} H_1$ and $G\underline{\vee} H_2$ are $L$-cospectral.
  \item[\rm (c)]  If $G_1$ and $G_2$ are $L$-cospectral regular graphs, and $H_1$ and $H_2$ are $L$-cospectral graphs, then $G_1\underline{\vee} H_1$ and $G_2\underline{\vee} H_2$ are $L$-cospectral.
\end{itemize}
\end{cor}

\subsection{$Q$-spectra of subdivision-edge joins}
\label{sec:SEQ}

\begin{thm}\label{SEQthm1}
Let $G_1$ be an $r_1$-regular graph on $n_1$ vertices and $m_1$ edges, and $G_2$ an arbitrary graph on $n_2$ vertices. Then
\begin{eqnarray*}
\phi\left(Q(G_1 \underline{\vee}  G_2);x\right) &=& (x-2-n_2)^{m_1-n_1}\cdot\prod_{i=1}^{n_2}\Big(x-m_1-\nu_i(G_2)\Big)\\
&& \cdot\prod_{i=1}^{n_1-1}\Big(x^2-(2+r_1+n_2)x+r_1n_2+2r_1-\nu_i(G_1)\Big)\\
&&\cdot\Big(x^2-(2+r_1+n_2)x+r_1n_2-m_1(x-r_1)\cdot\Gamma_{Q(G_2)}(x-m_1)\Big).
\end{eqnarray*}
\end{thm}

\begin{proof}
Let $R$ be the incidence matrix of $G_1$. Then the signless Laplacian matrix of $G_1\underline{\vee} G_2$ can be written as
\[Q(G_1\underline{\vee} G_2)=\bmat{
                          r_1I_{n_1} & R & 0_{n_1\times n_2} \\[0.2cm]
                          R^T & (2+n_2)I_{m_1} & J_{m_1\times n_2}\\[0.2cm]
                          0_{n_2\times n_1} & J_{n_2\times m_1} & m_1I_{n_2}+Q(G_2)
                       }.\]
The rest of the proof is similar to that of Theorem \ref{SELthm1} and hence we omit details.
\qed
\end{proof}

By applying (\ref{eq:GammaT}) again, Theorem \ref{SEQthm1} implies the following result.
\begin{cor}\label{SEQcor1}
Let $G_1$ be an $r_1$-regular graphs on $n_1$ vertices and $m_1$ edges, and $G_2$ an $r_2$-regular graphs on $n_2$ vertices. Then
\begin{eqnarray*}
\phi\left(Q(G_1\dot{\vee} G_2);x\right)&=&(x-2-n_2)^{m_1-n_1}\cdot\Big(x^3-ax^2+bx-2r_1r_2n_2\Big)\cdot\prod_{i=1}^{n_2-1}\Big(x-m_1-\nu_i(G_2)\Big)\\
&&\cdot\prod_{i=1}^{n_1-1}\Big(x^2-(2+r_1+n_2)x+r_1n_2+2r_1-\nu_i(G_1)\Big).
\end{eqnarray*}
where $a=2+2r_2+r_1+m_1+n_2$ and $b=r_1n_2+2m_1+r_1m_1+4r_2+2r_1r_2+2r_2n_2$.
\end{cor}

Finally, Theorem \ref{SEQthm1} enables us to construct infinitely many pairs of $Q$-cospectral graphs.

\begin{cor}
\label{cor:SEQcosp}
\begin{itemize}
  \item[\rm (a)]  If $G_1$ and $G_2$ are $Q$-cospectral regular graphs, and $H$ is a regular graph, then $G_1\underline{\vee} H$ and $G_2\underline{\vee} H$ are $Q$-cospectral.
  \item[\rm (b)]  If $G$ is a regular graph, and $H_1$ and $H_2$ are $Q$-cospectral graphs with $\Gamma_{Q(H_1)}(x)=\Gamma_{Q(H_2)}(x)$, then $G\underline{\vee} H_1$ and $G\underline{\vee} H_2$ are $Q$-cospectral.
\end{itemize}
\end{cor}

\end{document}